\documentclass[10pt]{article}

\usepackage[T1]{fontenc}
\usepackage{amsmath,amsthm,amssymb}
\usepackage{graphicx}
\usepackage{tikz}
\usetikzlibrary{arrows.meta,decorations.pathreplacing}

\theoremstyle{plain}
\newtheorem{thm}{Theorem}[section]
\newtheorem{lem}[thm]{Lemma}
\newtheorem{prop}[thm]{Proposition}
\newtheorem{cor}[thm]{Corollary}

\newtheorem{obs}[thm]{Observation}

\theoremstyle{definition}
\newtheorem{defn}{Definition}[section]

\newcommand{\sub}{\textsc{SUB}$_{<tu}$}

\newcommand{\NN}{\mathbb{N}}

\providecommand{\keywords}[1]
{
  \small
  \textbf{\textit{Keywords:}} #1
}

\providecommand{\classification}[1]
{
  \small
  \textbf{\textit{Classification:}} #1
}

\title{Duality, $\chi^<$-Boundedness and Order Density of Ordered Graphs}

\author{
Michal \v{C}ert\'{\i}k, Jaroslav Ne\v{s}et\v{r}il \\ 
Computer Science Institute, Faculty of Mathematics and Physics \\
Charles University\\
Prague, Czech Republic \\
}

\date{}

\begin{document}

\maketitle

\begin{abstract}

  We characterize the singleton duality pairs for ordered graphs. We also define a corresponding notion of $\chi^<$-boundedness for ordered graphs, show that all ordered graphs are $\chi^<$-bounded, and prove an analogue of the Gy\'arf\'as-Sumner conjecture for ordered graphs. 
  We also prove an analogue of the Sparse Incomparability Lemma for ordered graphs. 
  We then use this result to show that classes of ordered graphs form a dense order under ordered homomorphisms. We also show that, compared to unordered graphs, ordered graphs have more gaps, defined by consecutive monotone matchings and by even more generic pairs of ordered graphs differing by one isolated edge.
  
\end{abstract}


\classification{05C60, 06D50}

\section{Introduction}


All graphs in this paper are finite, simple and nonempty, unless an empty subgraph is explicitly allowed.

An \emph{ordered graph} is an undirected graph whose vertices are totally ordered. Thus, the ordered graph $G$ is a triple $G = (V,E,\le_G)$ (see Figure~\ref{fig:OrdHomsInterval}).



In this paper, we consider the homomorphisms of ordered graphs. These are defined as edge- and order-preserving mappings and they are naturally related to ordered chromatic number (which in turn naturally relates to extremal results, see e.g.~\cite{Pach2006}).

For ordered graphs $G=(V,E,\le_G)$ and $G'=(V',E',\le_{G'})$, an \emph{ordered homomorphism} is a mapping $f:V\to V'$ preserving both edges and orderings. Explicitly, $f$ satisfies
\begin{enumerate}
    \item $f(u)f(v) \in E'$ for all $uv \in E$,
    \item $f(u) \le_{G'} f(v)$ whenever $u \le_{G} v$.
\end{enumerate}

The existence of ordered homomorphism between ordered graphs $G$ and $H$ will be denoted by $G\to H$ (see Figure~\ref{fig:OrdHomsInterval}).

\begin{figure}[ht]
\begin{center}
\begin{tikzpicture}[x=0.90cm,y=0.90cm]
  \tikzset{
    ordgraphline/.style={line width=0.65pt,line cap=round},
    ordtick/.style={line width=0.6pt},
    ordintervalbrace/.style={
      decorate,decoration={brace,mirror,amplitude=4pt}
    },
    ordmaparrow/.style={
      -{Latex[length=7pt,width=5.5pt]},line width=0.75pt
    }
  }

  \node[font=\LARGE] at (0.72,10.28) {$G$};
  \draw[ordgraphline] (-0.15,8.25) -- (7.55,8.25);
  \foreach \x/\lab in {1.00/1,2.10/2,3.20/3,4.30/4,5.40/5,6.50/6}{
    \draw[ordtick] (\x,7.93) -- (\x,8.68);
    \node[below=3pt,font=\normalsize] at (\x,7.93) {$v_{\lab}$};
  }
  \draw[ordgraphline]
    (1.00,8.25)
    arc[start angle=180,end angle=0,x radius=2.75,y radius=2.25];
  \draw[ordgraphline]
    (1.00,8.25)
    arc[start angle=180,end angle=0,x radius=1.10,y radius=1.05];
  \draw[ordgraphline]
    (2.10,8.25)
    arc[start angle=180,end angle=0,x radius=1.65,y radius=1.72];
  \draw[ordgraphline]
    (3.20,8.25)
    arc[start angle=180,end angle=0,x radius=1.65,y radius=1.72];
  \draw[ordgraphline]
    (3.20,8.25)
    arc[start angle=180,end angle=0,x radius=0.55,y radius=0.55];

  \draw[ordintervalbrace] (0.82,7.32) -- (2.28,7.32)
    node[midway,below=7pt,font=\normalsize]
    {$\text{Interval }[1,2]$};
  \draw[ordintervalbrace] (4.12,7.32) -- (6.68,7.32)
    node[midway,below=7pt,font=\normalsize]
    {$\text{Interval }[4,6]$};

  \draw[ordmaparrow] (3.75,6.35) -- (3.75,4.55)
    node[midway,right=8pt,font=\Large] {$f$};

  \node[font=\LARGE] at (1.05,4.10) {$H$};
  \draw[ordgraphline] (-0.15,2.40) -- (7.55,2.40);
  \foreach \x in {1.25,3.75,6.25}{
    \draw[ordtick] (\x,2.08) -- (\x,2.83);
  }
  \draw[ordgraphline]
    (1.25,2.40)
    arc[start angle=180,end angle=0,x radius=2.50,y radius=1.72];
  \draw[ordgraphline]
    (1.25,2.40)
    arc[start angle=180,end angle=0,x radius=1.25,y radius=1.02];
  \draw[ordgraphline]
    (3.75,2.40)
    arc[start angle=180,end angle=0,x radius=1.25,y radius=1.02];

  \node[below=4pt,font=\normalsize] at (3.75,2.08)
    {$v'_2=f(v_3)$};
  \node[below=16pt,font=\normalsize] at (1.25,2.08)
    {$v'_1=f(v_1)=f(v_2)$};
  \node[below=16pt,font=\normalsize] at (6.25,2.08)
    {$v'_3=f(v_4)=f(v_5)=f(v_6)$};
\end{tikzpicture}
\end{center}
\caption{Ordered Homomorphism $f$ and Independent Intervals.}
\label{fig:OrdHomsInterval}
\end{figure}

\section{Motivation}

Ordered graphs commonly appear in many different settings: extremal theory~\cite{Pach2006,conlon2016ordered}, category theory~\cite{Hedrln1967,nesetril2016characterization}, Ramsey theory~\cite{Nesetril1996,Hedrln1967,balko2022offdiagonal,Balko_2020}, among others. Recently, it has been shown that the concept of twin width in graphs corresponds to NIP ("not the independence property") classes of ordered graphs~(\cite{bonnet2021twinwidthIV, bonnet2024twinwidth}), thereby linking graph theory with model theory.


In relation to the aforementioned research, homomorphisms of ordered graphs provide both validation and extension: although they impose stricter conditions in comparison to standard homomorphisms (see, for example, \cite{HellNesetrilGraphHomomorphisms}), they also exhibit their own unique complexity (see, for instance, \cite{AxenovichChromaticNumber2018,Guerra2012,braun2013cellular-998,nie2023asymptotic-fa9,bose2004ordered-e8c,certikorderedgraphsphd2025}). Ordered homomorphisms compose, and thus most of the categorical definitions can be considered without any changes; see, e.g.~\cite{HellNesetrilGraphHomomorphisms}.

The exploration of complexities and parameterized complexities concerning ordered graphs and their homomorphisms has also been examined from multiple perspectives. Recently, \cite{kun2025dichotomy-dd6} has shown that ordering problems for graphs defined by finitely many forbidden ordered subgraphs capture the class $\mathbf{NP}$. In \cite{Duffus1995}, the complexities of decision problems involving ordered graphs and their subgraphs are studied. We address in \cite{certik_complexity_2025,certik_core_2025, certik_matching_2025} the complexities and parameterized complexities of fundamental problems related to ordered graphs, their homomorphisms, and their cores.




\section{Statement of Results}

In this article, we characterize homomorphism dualities and we also consider Gy\'arf\'as-Sumner type problems ($\chi$-boundedness), which in the context of ordered graphs we fully characterize. It is perhaps surprising that these questions, which are difficult for unordered graphs and homomorphisms, find a simple and transparent solution for ordered graphs and their homomorphisms. 

We also examine the order density defined by homomorphisms of ordered graphs. For these purposes, we prove and apply the Sparse Incomparability Lemma for ordered graphs. We will also show that, compared to the order density defined by homomorphisms of unordered graphs, the order defined by ordered homomorphisms has many more gaps.

Section~\ref{sec:Preliminaries} focuses on introducing essential definitions and presents some of the foundational results on the coloring of ordered graphs (analogous to the coloring of unordered graphs). This is proved to be feasible in polynomial time (in fact linear) for ordered graphs. 

In Section~\ref{chapt:Dualities}, Theorem~\ref{thm:Uniq}, we provide the duality result in the ordered setting. The key role is played by monotone matching. Note that monotone matching also plays a role in the Ramsey context (see~\cite{balko2023ordered}).

In Section~\ref{chapt:ChiBound} we deal with questions which subgraphs are unavoidable in a large chromatic number.

Formulating it dually, we ask when the chromatic number of a graph is bounded as a function of its subgraphs (for unordered graphs this amounts to bounding the chromatic number as a function of the clique number, which leads to $\chi$-bounded classes and the Gy\'arf\'as-Sumner conjecture).

For ordered graphs, the situation is easier, and we prove the analogous statement, Theorem~\ref{thm:Gyarfas}, with the following definition of some such unavoidable graphs (see also Figure~\ref{fig:MnMLR}).

\begin{defn}
    \emph{Monotone matching} $M_n$ has points $a_i, b_i, i=1,\ldots, n$, with ordering $a_1<b_1<a_2<b_2<\ldots<a_n<b_n$ and edges $\{a_i,b_i\}, i=1,\ldots,n$. $a_i$ are left vertices, $b_i$ are right vertices.
    
\begin{itemize}
    \item $M^{LR}_n$ is $M_n$ together with all edges $\{a_i, b_j\}, i<j$.
    \item $M^{RL}_n$ is $M_n$ together with all edges $\{b_i, a_j\}, i<j$.
    \item $M^{+}_n$ is just $M^{LR}_n \cup M^{RL}_n$.
\end{itemize}
\end{defn}

\begin{figure}[!htbp]
\begin{center}
{%
\tikzset{
  forbgraphline/.style={line width=0.65pt,line cap=round},
  forbtick/.style={line width=0.6pt}
}

\newcommand{\forbedgearc}[3]{%
  \pgfmathsetmacro{\forbedgeheight}{#3+0.60*(#2-#1)}%
  \draw[forbgraphline] (#1,#3)
    .. controls (#1,\forbedgeheight) and (#2,\forbedgeheight)
    .. (#2,#3);
}

\newcommand{\forbmatching}[1]{%
  \forbedgearc{2.00}{2.65}{#1}%
  \forbedgearc{3.30}{3.95}{#1}%
  \forbedgearc{4.60}{5.25}{#1}%
  \forbedgearc{7.20}{7.85}{#1}%
}

\newcommand{\forblrconnections}[1]{%
  \forbedgearc{2.00}{3.95}{#1}
  \forbedgearc{2.00}{5.25}{#1}
  \forbedgearc{2.00}{7.85}{#1}
  \forbedgearc{3.30}{5.25}{#1}
  \forbedgearc{3.30}{7.85}{#1}
  \forbedgearc{4.60}{7.85}{#1}
}

\newcommand{\forbrlconnections}[1]{%
  \forbedgearc{2.65}{3.30}{#1}
  \forbedgearc{2.65}{4.60}{#1}
  \forbedgearc{2.65}{7.20}{#1}
  \forbedgearc{3.95}{4.60}{#1}
  \forbedgearc{3.95}{7.20}{#1}
  \forbedgearc{5.25}{7.20}{#1}
}

\newcommand{\forbpanelvertices}[1]{%
  \draw[forbgraphline] (1.35,#1) -- (8.55,#1);
  \foreach \x/\lab in {
    2.00/{a_1},2.65/{b_1},3.30/{a_2},3.95/{b_2},
    4.60/{a_3},5.25/{b_3},7.20/{a_n},7.85/{b_n}
  }{
    \draw[forbtick] (\x,{#1-0.28}) -- (\x,{#1+0.28});
    \node[below=2pt,font=\scriptsize] at (\x,{#1-0.28}) {$\lab$};
  }
  \node[font=\normalsize] at (6.22,{#1+0.24}) {$\cdots$};
}

\begin{tikzpicture}[x=1cm,y=1cm]
  \node[anchor=west,font=\Large] at (0.15,11.95) {$M_n$};
  \forbmatching{11.45}
  \forbpanelvertices{11.45}

  \node[anchor=west,font=\Large] at (0.15,8.15) {$M_n^{LR}$};
  \forbmatching{7.75}
  \forblrconnections{7.75}
  \forbpanelvertices{7.75}

  \node[anchor=west,font=\Large] at (0.15,4.45) {$M_n^{RL}$};
  \forbmatching{4.05}
  \forbrlconnections{4.05}
  \forbpanelvertices{4.05}

  \node[anchor=west,font=\Large] at (0.15,0.75) {$M_n^{+}$};
  \forbmatching{0.35}
  \forblrconnections{0.35}
  \forbrlconnections{0.35}
  \forbpanelvertices{0.35}
\end{tikzpicture}%
}
\end{center}
\caption{$M_n, M^{LR}_n, M^{RL}_n$ and $M^{+}_n$}
\label{fig:MnMLR}
\end{figure}

\begin{thm}
\label{thm:Gyarfas}
Let $G$ be an ordered graph that does not contain any of the following graphs as induced subgraphs:
\[
K_m, M_n, M^{LR}_k, M^{+}_l, \quad n,k\ge 2, \; m,l\ge 3.
\]

Then there exists $f(k,l,m,n):\mathbb{N}^4\to\mathbb{N}$ such that $\chi^<(G)\le f(k,l,m,n)$.
\end{thm}

In Section~\ref{sec:sparse}, we prove an analogue of the Sparse Incomparability Lemma for ordered graphs (see, e.g., \cite{HellNesetrilGraphHomomorphisms,d0f6743d-1cbb-3b59-aa4a-4a1f20ecc7bb,NESETRIL1989133,NESETRIL2004161} for unordered graphs).

In Section~\ref{sec:density}, we then use this result to show that no gap has an upper ordered core whose connected components all have order at least three. We also show some other classes of ordered graphs defining a dense order, and that pairs of consecutive monotone matchings and certain ordered graphs differing by one isolated edge define gaps in this order (gap in the sense of two ordered graphs $G$ and $H, G<H$, for which there is no ordered graph $F$, such that $G<F<H$, with $<$ defined by ordered homomorphism order, where we denote $G<H$ for $G\to H$ and $H\not\to G$).

\section{Ordered Coloring}
\label{sec:Preliminaries}

Let us start this section with a definition of an \emph{independent interval}, which will be a set of independent vertices; explicitly, if $\le_G$ is given as $v_1, v_2, \ldots, v_n$, then an $[i,j]$ independent interval in $G$ is the set $\{v_i,v_{i+1},\ldots,v_j\}$ that does not contain any edge of $G$ (see Figure~\ref{fig:OrdHomsInterval}).

We then define a \emph{double} as a pair of consecutive vertices $v_i$ and $v_{i+1}$, connected by an edge, in ordered graphs. 

Next, we define a \emph{(singleton) homomorphism duality} as a pair of ordered graphs $F, D$ that satisfy
\[
F\not\to G \text{ if and only if } G\to D
\]
for every ordered graph $G$.

For graphs and relational structures, the dualities are characterized in~\cite{NESETRILTARDIF200080}.

An \emph{ordered core} of an ordered graph $G$ will then be defined as the smallest ordered subgraph $H$ of $G$ such that $G\to H$. (Equivalently, this is the smallest ordered retract of $G$.)

Let us continue this section with a definition of \emph{(ordered) chromatic number} $\chi^<(G)$ being the minimum $k$ such that $V(G)$ can be partitioned into $k$ disjoint independent intervals. Notice that for ordered graphs this is the order of the smallest homomorphic image and, alternatively, the minimum $k$ such that $G\to K_k$. ($K_k$ being a complete graph with fixed linear ordering.) 

Unlike the chromatic number of (unordered) graphs, the ordered chromatic number can be determined by a simple greedy algorithm. This is formulated in the following as Proposition~\ref{prop:GreedyAlg} (which may be folklore, possibly first observed in \cite{Pach2006}).

\emph{Greedy Algorithm} is a natural one: Process the vertices in the given order and color each vertex by the smallest available color that fits the rule.
What is the rule? For graph $G$ with ordering $\le_G$ each color of previously colored vertices has to be an independent interval in $\le_G$.

We have the following:

\begin{prop} \label{prop:GreedyAlg}
For every ordered graph $G$, the greedy algorithm finds $\chi^<(G)$ in polynomial time.
\end{prop}

\begin{proof}
    Put $\chi^<(G)=k$. Obviously, by the algorithm, the greedy algorithm finds a $k'$-coloring for $k'\ge k$.

    We prove $k'=k$ by induction on $k$. For $k=1$, the graph is just a single independent interval.

    In the induction step, let $\chi^<(G)=k+1$ and let $I_1, I_2, \ldots, I_{k+1}$ be an optimal coloring. Consider the independent interval $I'_1$ given by the greedy algorithm. Clearly $I_1\subseteq I'_1$ and thus $G'=G-I'_1$ (with the vertex set $I'_1$ deleted) satisfies $\chi^<(G')=k$ (it cannot be smaller than $k$ since then $\chi^<(G)$ would be smaller than $k+1$) and the greedy algorithm also produces a $k$-coloring. Thus, also $k+1$ is the result of the greedy algorithm.
\end{proof}

It follows that for any $G$, determining $\chi^<(G)$ is in $\mathcal{P}$, $\mathcal{P}$ being a class of decision problems that can be solved in polynomial time. The Greedy algorithm goes over all the vertices of $G$ and checks at each step whether the vertex is not connected to some of the previous vertices (with respect to the ordering of $G$), therefore, the complexity of an algorithm is at most $\mathcal{O}(|V(G)|^2)$. Alternatively, as at each vertex $v$ we look back at vertices connected to $v$, and we do not consider any of the edges more than once, the complexity of the Greedy algorithm is (linear) $\mathcal{O}(|V(G)| + |E(G)|)$.

Let us continue this section by defining $P_m$ as an ordered graph on the $m$ vertices with natural ordering and $E(P_m)=\{\{v_iv_{i+1}\}\mid i=1,\ldots,m-1;v_i\in V(P_m)\}$. We will call $P_m$ a \emph{directed path}. Notice that this definition differs from the definition of path for unordered graphs, where $m$ denotes the number of edges instead of the number of vertices. We also define an edge set $E'$ in an ordered graph as \emph{non-intersecting} if for every two distinct edges $e_1=\{v_1,v_2\}, v_1<v_2$ and $e_2=\{v_3,v_4\}, v_3<v_4$ in $E'$, either $v_2\le v_3$ or $v_4\le v_1$.

\begin{obs}
\label{Obs:ConsEdges}
    Let $G$ be an ordered graph with $k$ non-intersecting edges. Then if there exists an ordered homomorphism from $G$ to $H$, $H$ must also contain at least $k$ non-intersecting edges.
\end{obs}

\begin{proof}
    By definition, the ordered homomorphism from $G$ to $H$ must preserve the ordering $<_G$ and the edges of $G$, therefore the observation follows.
\end{proof}

Let $\lambda(G)$ be the maximum number of non-intersecting edges in $G$, then it follows that $\lambda(G)$ is monotone invariant for ordered homomorphisms. $\lambda(G)$ for ordered graphs here plays the role of $\omega(G)$ for unordered graphs.

The following observation is also not difficult to see. We define $A^L_m(G)$ as an ordered graph on $m=\chi^<(G)$ vertices, resulting from running the Greedy Algorithm on an ordered graph $G$.

\begin{obs}
\label{obs:AgPgGreedy}
    Let $G$ be an ordered graph. 
    Then a directed path $P_{m}$ is an ordered subgraph of $A^L_m(G)$.
\end{obs}

\begin{proof}
    We shall prove this by contradiction. Assume $A^L_m(G)$ does not contain a directed path $P_{m}$. Then there exist two independent vertices $v_i$ and $v_{i+1}, i\in [m-1]$ in $A^L_m(G)$. But then these vertices should have been mapped to one vertex, a contradiction.
\end{proof}

We will now prove another related result that will be useful later on. We denote $\mathcal{A}_m(G)$ the set of graphs on $m=\chi^< (G)$ vertices, where for every $A_m(G) \in \mathcal{A}_m(G)$, there exists an ordered homomorphism from $G$ to $A_m(G)$.

\begin{lem}
\label{lem:NonIntersectEdgesinGisAm}
Let $G$ be an ordered graph, $A_m(G) \in \mathcal{A}_m(G)$, and $M_k$ be monotone matchings. Then if there exists an ordered homomorphism $M_k\to A_m(G)$, then there is an ordered homomorphism $M_k\to G$.
\end{lem}

\begin{proof}
    We know from the previous Observation~\ref{obs:AgPgGreedy}, that $A_m(G)$ contains $P_m$, therefore $A_m(G)$ contains $m-1$ non-intersecting edges. We also know from Observation~\ref{Obs:ConsEdges}, that if $M_k\to A_m(G)$, then $A_m(G)$ must contain at least $k$ non-intersecting edges, since all $k$ edges in $M_k$ are non-intersecting, and $k$ edges in monotone matching can always map to any ordered graph with at least $k$ non-intersecting edges.

    We prove this statement by showing that if $A_m(G) \in \mathcal{A}_m(G)$, then $G$ must also contain $m-1$ non-intersecting edges. Then if $M_k\to A_m(G)$, then $M_k$ can also map to these $m-1$ non-intersecting edges in $G$.
    
    For brevity, and without risk of confusion, we will denote $A_m(G)$ simply as $A_m$.

    W.l.o.g., let us assume that $k=m-1$, since if $k<m-1$, then the ordered homomorphism $M_k\to G$ will be easier to find, and if $k>m-1$, then, from Observation~\ref{Obs:ConsEdges}, $M_k\not\to A_m$.



    Let $A_{m}$ be produced by the Greedy Algorithm. 
    We notice that by the nature of the Greedy algorithm, for each two consecutive intervals $I_i=[j_1,j_2]$ and $I_{i+1}=[j_2+1,j_3],i\in[m-1]$ of $G$ produced by this algorithm, vertex $j_2+1$ of interval $I_{i+1}$ is connected to some of the vertices in the interval $I_i$. Otherwise, the Greedy algorithm would include this vertex in interval $I_i$.

    For each interval $I_i,i=2,3,\ldots,m$ of $G$ produced by this algorithm, we will therefore choose the edge connecting its first vertex with some vertex in the previous interval, getting $m-1$ non-intersecting edges in $G$.
    
    To see that these $m-1$ edges in $G$ are non-intersecting, assume that there exist two consecutive intervals $I_i=[j_1,j_2]$ and $I_{i+1}=[j_2+1,j_3],2\le i\le m-1$ of $G$ produced by the algorithm, where their respective chosen edges intersect: edge $e_i$ chosen for the interval $I_i$ connecting the vertex $j_1$ with vertex in interval $I_{i-1}$ and the edge $e_{i+1}$ chosen for the interval $I_{i+1}$ connecting the vertex $j_2+1$ with vertex in interval $I_{i}$. Since the edge $e_i$ connects the vertex $j_1$ with a vertex in interval $I_{i-1}$, and the edge $e_{i+1}$ connects the vertex $j_2+1$ with some vertex in interval $I_{i}=[j_1,j_2]$, the only way these edges can intersect is for the edge $e_{i+1}$ to be incident to some vertex $v\in I_i$, such that $v<j_1$. But the edge $e_{i+1}$ is incident with vertex $v\ge j_1$, a contradiction. 
    
    Therefore there are $m-1$ non-intersecting edges in $G$ where $M_{k}$ can map.

\end{proof}

\section{Dualities of Ordered Graphs}
\label{chapt:Dualities}

For unordered graphs and unordered relational structures, the dualities are characterized in~\cite{NESETRILTARDIF200080}. Additional results on the topic can also be found in, for example, \cite{nesetriltardifdual2005}, \cite{nesetril2021fromsparse}, and \cite{Foniok2008}.

In this section, we provide the characterization in the ordered setting. The key role is played by ordered monotone matching. Note that matchings also play a role in the ordered Ramsey context (see, e.g., \cite{balko2023ordered}).

We will show that the following pairs of ordered cores are the only singleton homomorphism dualities of ordered graphs.



\begin{thm}
\label{thm:Uniq}
The only pair of ordered cores $(F,H)$, satisfying $F\not \to G$ if and only if $G\to H$, for any ordered graphs $G$, is $(F,H) = (M_k, K_k), k\in\NN$.
\end{thm}

\begin{proof}

We prove this statement in two steps. First, we show that ordered cores $M_k$ and $K_k$ form a singleton homomorphism duality pair, and then we prove that this is the only singleton duality for ordered graphs.

We first show that if $M_k\to G$, then $G\not\to K_k$. From the Observation~\ref{Obs:ConsEdges}, we see that $M_k\not\to K_k$, since $\lambda(K_k)=k-1$ and $\lambda(M_k)=k$. Therefore, the first implication holds by transitivity of existence of homomorphisms.

On the other hand, if $M_k\not\to G$, then we use the greedy algorithm on $G$ and Proposition~\ref{prop:GreedyAlg} to obtain an ordered graph $A_g$ with a minimum number of $g$ vertices such that $G\to A_g$. From Observation~\ref{obs:AgPgGreedy}, we know that $A_g$ contains a directed path $P_{g}$.

It is clear that $g-1<k$, as otherwise $M_k\to P_g\subseteq A_g$, and if $M_k\to A_g$ then $M_k$ can also map to $G$. This can be seen again using Observation~\ref{Obs:ConsEdges}, since if $g-1\ge k$, then $M_k\to A_g$ can map the edges of $M_k$ to the $g-1$ non-intersecting edges of $P_g$ in $A_g$, therefore, from Lemma~\ref{lem:NonIntersectEdgesinGisAm}, $M_k$ could map to the $g-1$ non-intersecting edges in $G$. 

Hence, since $G\to A_g \to K_k$ (since $A_g$ is on $k$ or fewer vertices), we get $G\to K_k$.




%

We now show that $(M_k,K_k)$ is the only singleton homomorphism duality pair of ordered graphs.
Let $F$ and $H$ be ordered cores that for every ordered graph $G$, satisfy $F\not\to G$ if and only if $G\to H$.

Let $G$ be a monotone matching. We see that if $F\to G$, then $F$ must contain an independent interval partition that maps to a monotone matching. But since $F$ is a core, $F$ must be a monotone matching.

Now, let us assume $F\not\to G$. But then $G\not\to H$ for any sufficiently large monotone matching $G$.



Let us now assume that $M_k\not\to G$ if and only if $G\to H$, for some ordered core $H$. But substituting $K_k$ and $H$ for $G$, we get $M_k\not\to K_k$ if and only if $K_k\to H$ and $M_k\not\to H$ if and only if $H\to K_k$, respectively. However, since $H$ and $K_k$ are cores, and from \cite{certik_core_2025} we know that ordered core of an ordered graph is unique, $H$ must be isomorphic to $K_k$.

\end{proof}

Once again drawing on parallels with unordered graphs, we may note that a compelling extension of this study might be investigating duality pairs for different classes of ordered graphs. We will not address it in this work.

\section{$\chi^<$-boundedness of Ordered Graphs}
\label{chapt:ChiBound}

For unordered graphs, we recall that the $\chi$-bounded family $\mathcal{F}$ of graphs is one for which there is a function $f: \NN\to\NN$ such that, with $\omega(G)$ being an order of maximum clique in $G\in \mathcal{F}$, for every $\omega(G)$, $\chi(G)$ is at most $f(\omega(G))$ (see, e.g., \cite{gyrfs1987problems-b30}).

Motivated by the Duality Theorem~\ref{thm:Uniq}, 
let us define the following measure.

\begin{defn}[$\eta(G)$]
\hfill

     Let $G$ be an ordered graph, then $\eta(G)$ is the size of the maximum ordered monotone matching subgraph (not necessarily induced) $M_{\eta(G)}$ of $G$.
\end{defn}

We now show that all ordered graphs are $\chi^<$-bounded (using $\eta(G)$ instead of $\omega(G)$).

\begin{thm}
\label{thm:chiboundedness}
Let $G$ be an ordered graph, then $\chi^< (G)\le 2\eta(G)+1$.
\end{thm}

\begin{proof}

Let us choose a monotone ordered matching subgraph $M_{\eta(G)}$ of $G$, and let us run the Greedy Algorithm on $G$. From Proposition~\ref{prop:GreedyAlg}, as a result of this procedure, we will get an ordered graph $A_g$, where $g=\chi^<(G)$. 

Now, for a contradiction, let us assume that $\chi^<(G)>2\eta(G)+1$. Then, from Proposition~\ref{prop:GreedyAlg} and Observation~\ref{obs:AgPgGreedy}, $A_g$ will have at least $2\eta(G)+2$ vertices and contain $P_{2\eta(G)+2}$, respectively. By Lemma~\ref{lem:NonIntersectEdgesinGisAm}, $G$ has $2\eta(G)+1$ non-intersecting edges; taking every second edge gives a monotone matching subgraph of size $\eta(G)+1$, a contradiction.
\end{proof}


We say that $\mathcal{F}$ is a $\chi^<$\emph{-bounded} family of ordered graphs, if there exists a function $f$ such that for every $G\in \mathcal{F}$, $\chi^<(G)$ is at most $f(\eta(G))$.


Let us now try to prove a stronger (induced) version of the statement, again borrowing an idea from unordered graphs --- the \emph{Gy\'arf\'as--Sumner conjecture} (see~\cite{GyarfasOnRamseyCoveringNumbers,Sumner1981}). The conjecture states that for every tree $T$ and complete graph $K$, the graphs with neither $T$ nor $K$ as induced subgraphs can be properly colored using only a constant number of colors, depending on $T$ and $K$ only.

We shall replace the tree with the forbidden structures introduced in Definition~\ref{OrdMonMatching} (see also Figure~\ref{fig:MnMLR}).

\begin{defn}
\label{OrdMonMatching}
    \emph{Ordered Monotone Matching} $M_n$ has vertices $a_i, b_i, i=1,\ldots, n$, with ordering $a_1<b_1<a_2<b_2<\ldots<a_n<b_n$ and edges $\{a_i,b_i\}, i=1,\ldots,n$. $a_i$ are \emph{left} vertices, $b_i$ are \emph{right} vertices.
    Further, we define $M^{LR}_n$ as $M_n$ together with all edges $\{a_i, b_j\}, i<j$, $M^{RL}_n$ as $M_n$ together with all edges $\{b_i, a_j\}, i<j$, and $M^{+}_n$ as $M^{LR}_n \cup M^{RL}_n$.
\end{defn}




We now prepare the proof of Theorem~\ref{thm:Gyarfas}.
Advancing the proof, let us first define a set of \emph{incomparable} ordered graphs. 

\begin{defn}[\emph{Incomparable} Ordered Graphs]
\label{def:etaG}
\hfill

A set $\mathcal{G}$ of ordered graphs is \emph{incomparable}, if for every two distinct ordered graphs $G,H\in \mathcal{G}$, $G$ is not an induced subgraph of $H$, and $H$ is not an induced subgraph of $G$.

\end{defn}

For fixed $n,k\ge 2, m,l\ge 3,$ we prove that $M_n, K_m, M^{LR}_k, M^{+}_l$ are incomparable ordered graphs and we determine their chromatic number.


\begin{prop}
\label{prop:non-chibound}
Let $k,l,m,n\in \mathbb{N}$, $n,k\ge 2$, $m,l\ge 3$ be fixed. Then
\[
M_n,\ K_m,\ M^{LR}_k,\ M^{+}_l
\]
are incomparable ordered graphs, and $M^{LR}_k$ is an induced ordered subgraph of $M^{RL}_{2k}$. Moreover,
\[
\chi^<(M_n)=n+1,\quad
\chi^<(K_m)=m,\quad
\chi^<(M^{LR}_k)=k+1,\quad
\chi^<(M^{+}_l)=2l.
\]
\end{prop}

\begin{proof}

We start the proof by noting that, by definition, $M_n, M^{LR}_k$ and $M^{+}_l$ are bipartite graphs (with partition $\{a_i\}$ and $\{b_i\}$). Moreover, $M^{+}_l$ is a complete bipartite graph, and $M^{LR}_k$ is a half graph, where $a_i,b_j$ are connected if and only if $i\le j$.

Clearly, $K_m, M^{LR}_k, M^{+}_l$ cannot be induced subgraphs of $M_n$, since $M_n$ is disconnected and every induced subgraph of $M_n$ is a matching. Also, $M_n, M^{LR}_k, M^{+}_l$ cannot be induced subgraphs of $K_m$, since every induced subgraph of $K_m$ is complete, whereas these graphs are not complete. For the same reason, $K_m$ cannot be an induced subgraph of $M^{LR}_k$ or $M^+_l$, since both are bipartite and $m\ge 3$.

Since $M^+_l$ is complete bipartite, it does not contain $M_n$ or $M^{LR}_k$ as induced ordered subgraphs: an induced subgraph of $M^+_l$ is complete bipartite between the selected left and right vertices, whereas $M_n$ is a matching and $M^{LR}_k$ is a proper half graph.

So it remains only to show that $M^{LR}_k$ does not contain $M_n$ or $M^+_l$ as induced ordered subgraphs.

First, $M^{LR}_k$ does not contain an induced $M_n$ for $n\ge 2$. Indeed, suppose two non-intersecting edges of $M^{LR}_k$ are given by $a_i b_j$ and $a_p b_q$, with the first edge lying entirely before the second. Then $i\le j<p\le q$, and hence $a_i$ is also adjacent to $b_q$. Thus the two edges cannot form an induced matching of size two.

Second, $M^{LR}_k$ does not contain $M^+_l$ for $l\ge 3$. In $M^{LR}_k$, every edge $xy$ with $x<y$ has $x=a_i$ and $y=b_j$. Hence there are no three ordered vertices $x<y<z$ such that both $xy$ and $yz$ are edges. But $M^+_l$ contains the ordered path $P_{2l}$ on its consecutive vertices, so it cannot occur as an induced ordered subgraph of $M^{LR}_k$.

Let us now show that $M^{RL}_{2k}$ contains an induced ordered subgraph $M^{LR}_k$. This can be seen by selecting the vertices
\[
b_1,a_2,b_3,a_4,\ldots,b_{2k-1},a_{2k}
\]
of $M^{RL}_{2k}$. If these selected vertices are relabelled as
\[
a'_1,b'_1,a'_2,b'_2,\ldots,a'_k,b'_k,
\]
then $a'_p=b_{2p-1}$ and $b'_q=a_{2q}$ are adjacent precisely when
\[
2p-1\le 2q,
\]
equivalently when $p\le q$. This is exactly the edge relation of $M^{LR}_k$.

Finally, we compute the ordered chromatic numbers. For $M_n$, mapping each pair $b_i,a_{i+1}$, $1\le i\le n-1$, to one independent interval gives $\chi^<(M_n)=n+1$. Clearly $\chi^<(K_m)=m$. For $M^{LR}_k$, the independent intervals
\[
\{a_1\},\{b_1,a_2\},\{b_2,a_3\},\ldots,\{b_{k-1},a_k\},\{b_k\}
\]
show that $\chi^<(M^{LR}_k)\le k+1$, while the monotone matching $M_k\subseteq M^{LR}_k$ gives $\chi^<(M^{LR}_k)\ge k+1$. Thus $\chi^<(M^{LR}_k)=k+1$. Finally, $M^+_l$ contains the ordered path $P_{2l}$, so $\chi^<(M^+_l)=2l$.

\end{proof}

We are now ready to prove Theorem~\ref{thm:Gyarfas}.

\begin{proof}[Proof of Theorem~\ref{thm:Gyarfas}]

    Assume that Theorem~\ref{thm:Gyarfas} does not hold for $n,k\ge 2, m,l\ge 3$. Then it does not hold for $\Delta=\max\{m,n,l,2k\}$. Thus for every $N'=2N+1$, there exists an ordered graph $G$, not containing induced ordered subgraphs $K_m,M_n,M_k^{LR},M_l^+$, with chromatic number $\chi^<(G)\ge N'$ (we choose $N'=2N+1$ in order for $G$ to contain a monotone matching of size at least $N$, from Theorem~\ref{thm:chiboundedness}). We prove that this is a contradiction using Ramsey's Theorem.
    

    By Theorem~\ref{thm:Uniq} and Theorem~\ref{thm:chiboundedness}, $G$ contains a monotone matching subgraph $M_N$. Take the induced ordered subgraph $H_N$ of $G$ on the vertices of $M_N$.

    We see that there can be at most four different edges between two doubles of $H_N$ and we denote them as follows: 
    
    \begin{itemize}
        \item $LR$ edge, if it connects the left vertex of the first double with the right vertex of the second double.
        \item $RL$ edge, if it connects the right vertex of the first double with the left vertex of the second double.
        \item $LL$ edge, if it connects the left vertex of the first double with the left vertex of the second double.
        \item $RR$ edge, if it connects the right vertex of the first double with the right vertex of the second double.
    \end{itemize}
    
    We therefore have $16=|\mathcal{P}(\{LL,RR,LR,RL\})|$ possible ways in which two doubles can be connected.

    For disjoint doubles $e,e'$ of $H_N$, let $E(e,e')$ be the set of edges joining $e$ and $e'$ (i.e. one of $16$ possibilities).

    Let us take an ordered graph where all doubles are connected by the isomorphic edges (one of $16$). Then these $16$ different ordered graphs correspond to the following five induced subgraphs:

\begin{itemize}

    \item If $\Delta$ doubles are not connected by any edge, we get an induced monotone matching $M_\Delta$.

    \item If $\Delta$ doubles have $E(e,e')=\{LR\}$, we get an ordered graph $M^{LR}_\Delta$.
    
    \item If $\Delta$ doubles have $E(e,e')=\{RL\}$, we get an ordered graph $M^{RL}_\Delta$.

    \item If $\Delta$ doubles have $E(e,e')=\{LR, RL\}$, we get an ordered graph $M^{+}_\Delta$.

    \item In all other twelve cases, $E(e,e')$ contains $LL$ or $RR$ and thus we get a complete ordered graph $K_\Delta$.

\end{itemize}

     Ramsey's Theorem tells us that for any given finite number of colors, $c$, and any given integers $n_1,n_2,\ldots,n_c$, there is a number $R(n_1,n_2,\ldots,n_c)\in \mathbb{N}$, such that if the edges of a complete graph of order $R(n_1,n_2,\ldots,n_c)$ are colored with $c$ different colors, then for some $i\in [c]$, it must contain a complete subgraph of order $n_i$ whose edges are all of color $i$ (see, e.g., \cite{ramsey1930problem-202}).

Let $N=R(n_1,n_2,\ldots,n_c)=R(n_1,n_2,\ldots,n_{16})$ (since in our case $c=16$), where $n_k=\Delta$, $k\in[16]$. Consider doubles of $H_N$ as vertices (in other words, each double of $H_N$ being represented by a vertex), with $c=16$ types (or colors) of edges $E(e,e')$ (as defined above) between these doubles (represented by vertices), and apply Ramsey's Theorem. We obtain doubles $e_1,\ldots,e_{\Delta}$ in $H_N$, such that all 
edges between these doubles (represented by vertices) are of the same color $k$. Thus, the doubles $e_1,\ldots,e_{\Delta}$ of $H_N$ in $G$ 
contain one of the ordered induced subgraphs $K_{\Delta}, M_{\Delta}, M^{RL}_{\Delta}, M^{LR}_{\Delta}$ or $ M^{+}_{\Delta}$ of $G$, a contradiction with our choice of $\Delta$. Note that if the homogeneous type is $RL$, then the selected doubles induce $M^{RL}_\Delta$. Since $\Delta\ge 2k$ and $M^{LR}_k\subseteq_{\mathrm{ind}}M^{RL}_{2k}\subseteq_{\mathrm{ind}}M^{RL}_\Delta$, this gives an induced copy of $M^{LR}_k$.

    We showed in Proposition~\ref{prop:non-chibound} that for fixed $n,k\ge 2, m,l\ge 3$, the set of ordered graphs $M_n, K_m, M^{LR}_k, M^{+}_l$ is incomparable and that $M^{LR}_k$ is an induced subgraph of $M^{RL}_{2k}$.

    Therefore, the set of induced ordered subgraphs $M_n, K_m, M^{LR}_k, M^{+}_l$ of $G$ is indeed minimal and sufficient to limit the size of $\chi^<(G)$.

\end{proof}

This then implies an analogue of the Gy\'arf\'as--Sumner conjecture for $\chi^<$-boundedness of ordered graphs, replacing the forbidden structures (a tree and a clique for unordered graphs) in the original conjecture with our four graph classes.


\section{Sparse Incomparability Lemma for Ordered Homomorphisms}
\label{sec:sparse}

In this section, we examine an analogue of the Sparse Incomparability Lemma for ordered graphs. There are many applications of the Sparse Incomparability Lemma in the area of unordered graphs (see~\cite{HellNesetrilGraphHomomorphisms,d0f6743d-1cbb-3b59-aa4a-4a1f20ecc7bb,NESETRIL1989133,NESETRIL2004161}). We prove its analogue for ordered graphs and apply it in Section~\ref{sec:density} on order density of ordered homomorphisms.

Let us first define \emph{consecutive vertices} of $G$ as a set of vertices $i,i+1,\ldots,i+j\in V(G),i\in [n-1], j\in [n-i]$. Then if $G\to H$, we define \emph{gluing} the vertices or \emph{glued} vertices in $G$ as a set of (consecutive independent) vertices in $G$ which map to one vertex in $H$. The gluing notion will be used throughout, in particular in the proof of Lemma~\ref{thm:SparesIncomparabilityLemma} and in Section~\ref{sec:density}.

We also define an \emph{ordered matching} as an ordered graph $G$ where each vertex has at most one edge incident to it.

We shall now prove the following statement, which we call the Sparse Incomparability Lemma.

In the classical (unordered) setting, the Sparse Incomparability Lemma is a sparsification principle: given a graph $G$ and an integer $k$, one can construct a sparse graph $G'$ together with a homomorphism $G' \to G$ such that $G$ and $G'$ are indistinguishable with respect to homomorphisms to graphs on at most $k$ vertices, i.e., for every graph $H$ with $|V(H)|\le k$ we have $G \to H$ if and only if $G' \to H$. Informally, $G'$ serves as a sparse proxy for $G$ that preserves exactly the same small targets (and hence the same small obstructions). This lemma has many applications, for example to the structure of the homomorphism order and to transferring arguments from dense graphs to sparse ones.

In the ordered setting, we keep the same ``small-target indistinguishability'' viewpoint but compare ordered homomorphisms; moreover, the sparsification can be made particularly strong, as our construction yields $G'$ as an ordered matching.

\begin{lem}[Sparse Incomparability Lemma]
\label{thm:SparesIncomparabilityLemma}
    For any ordered graph $G$ and $k\in \mathbb{N}$, there exists an ordered matching $G'$ together with an ordered homomorphism $f\colon G'\to G$ such that for every ordered graph $H$ with $|H|\le k$, there is an ordered homomorphism $G'\to H$ if and only if there is an ordered homomorphism $G\to H$.
\end{lem}

\begin{figure}[ht]
\begin{center}
\begin{tikzpicture}[x=1cm,y=1cm]
  \tikzset{
    silmaparrow/.style={
      -{Latex[length=8pt,width=6pt]},line width=0.8pt
    },
    silobject/.style={font=\huge}
  }

  \node[silobject] (G)  at (1.20,5.50) {$G$};
  \node[silobject] (H)  at (7.25,5.50) {$H$};
  \node[silobject] (Gp) at (1.20,0.75) {$G'$};

  \draw[silmaparrow] (2.05,5.50) -- (6.25,5.50)
    node[midway,above=7pt,font=\huge] {$h$};
  \draw[silmaparrow] (1.20,1.75) -- (1.20,4.50)
    node[midway,left=8pt,font=\huge] {$f$};
  \draw[silmaparrow] (2.05,1.30) -- (6.42,4.63)
    node[pos=0.56,below right=4pt,font=\huge] {$g$};
\end{tikzpicture}
\end{center}
\caption{Sparse Incomparability Lemma Mapping}
\label{fig:SIL}
\end{figure}

\begin{proof}

We construct an ordered graph $G'$ and verify the required properties. Write $n=|G|$; the case $n=1$ is immediate, so assume $n\ge2$. Let $v_1, v_2, \ldots, v_n$ denote the vertices of $G$.

\textit{Construction of $G'$.}
Let $G'$ have $m = nk(n-1)$ vertices, partitioned into $n$ consecutive blocks $X_1, X_2, \ldots, X_n$, each of size $k(n-1)$:
\[
G' = (X_1, X_2, \ldots, X_n).
\]
Each block $X_i$ is further partitioned into $k$ consecutive sub-blocks $Y^1_i, Y^2_i, \ldots, Y^k_i$, each of size $n-1$:
\[
X_i = (Y^1_i, Y^2_i, \ldots, Y^k_i).
\]
The vertices of each sub-block $Y^j_i$ are indexed by $[n] \setminus \{i\}$ and ordered as:
\[
Y^j_i = (w^{ij}_1, w^{ij}_2, \ldots, w^{ij}_{i-1}, w^{ij}_{i+1}, \ldots, w^{ij}_n).
\]
For every edge $\{v_{i_1}v_{i_2}\} \in E(G)$ and every $j \in [k]$, we add the edge $\{w^{i_1j}_{i_2}w^{i_2j}_{i_1}\}$ to $G'$; there are no other edges. This completes the definition of $G'$ (see Figure \ref{pic:SpraseIncLemmaGraphG}).

\begin{figure}[p]
\centering
{%
\definecolor{edgeblue}{RGB}{63,81,181}
\definecolor{edgegreen}{RGB}{46,125,76}
\tikzset{
  silgraphline/.style={line width=0.65pt,line cap=round},
  siltick/.style={line width=0.6pt},
  silsmallbrace/.style={
    decorate,decoration={brace,mirror,amplitude=3.5pt}
  },
  sillargebrace/.style={
    decorate,decoration={brace,mirror,amplitude=5.5pt}
  }
}

\newcommand{\silsmallunderbrace}[4]{%
  \draw[silsmallbrace] (#1,2.08) -- (#2,2.08)
    node[midway,below=5pt,font=\normalsize] {$Y_{#3}^{#4}$};
}

\newcommand{\sillargeunderbrace}[3]{%
  \draw[sillargebrace] (#1,1.10) -- (#2,1.10)
    node[midway,below=9pt,font=\large] {$X_{#3}$};
}

\rotatebox{270}{%
\resizebox{0.7\textheight}{!}{%
\begin{tikzpicture}[x=0.67cm,y=0.67cm]
  \node[font=\LARGE] at (8.35,9.30) {$G$};
  \draw[silgraphline] (7.20,7.72) -- (10.05,7.72);
  \foreach \x/\lab in {7.45/1,8.62/2,9.79/3}{
    \draw[siltick] (\x,7.55) -- (\x,7.91);
    \node[below=2pt,font=\large] at (\x,7.55) {$v_{\lab}$};
  }
  \draw[silgraphline,edgeblue]
    (7.45,7.72)
    arc[start angle=180,end angle=0,x radius=0.585,y radius=0.60];
  \draw[silgraphline,edgegreen]
    (8.62,7.72)
    arc[start angle=180,end angle=0,x radius=0.585,y radius=0.60];

  \node[font=\LARGE] at (0.15,6.35) {$G'$};
  \draw[silgraphline] (-0.40,3.00) -- (17.10,3.00);

  \foreach \a/\b in {0.00/5.85,1.65/7.50,4.20/10.05}{
    \draw[silgraphline,edgeblue]
      (\a,3.00)
      arc[start angle=180,end angle=0,x radius=2.925,y radius=2.80];
  }
  \foreach \a/\b in {6.65/12.50,8.30/14.15,10.85/16.70}{
    \draw[silgraphline,edgegreen]
      (\a,3.00)
      arc[start angle=180,end angle=0,x radius=2.925,y radius=2.80];
  }

  \foreach \x/\sub/\sup in {
    0.00/2/{11},0.80/3/{11},
    1.65/2/{12},2.45/3/{12},
    4.20/2/{1k},5.00/3/{1k}
  }{
    \draw[siltick] (\x,2.80) -- (\x,3.20);
    \node[below=2pt,font=\small] at (\x,2.80) {$w_{\sub}^{\sup}$};
  }
  \node[font=\large] at (3.32,2.55) {$\cdots$};

  \foreach \x/\sub/\sup in {
    5.85/1/{21},6.65/3/{21},
    7.50/1/{22},8.30/3/{22},
    10.05/1/{2k},10.85/3/{2k}
  }{
    \draw[siltick] (\x,2.80) -- (\x,3.20);
    \node[below=2pt,font=\small] at (\x,2.80) {$w_{\sub}^{\sup}$};
  }
  \node[font=\large] at (9.17,2.55) {$\cdots$};

  \foreach \x/\sub/\sup in {
    11.70/1/{31},12.50/2/{31},
    13.35/1/{32},14.15/2/{32},
    15.90/1/{3k},16.70/2/{3k}
  }{
    \draw[siltick] (\x,2.80) -- (\x,3.20);
    \node[below=2pt,font=\small] at (\x,2.80) {$w_{\sub}^{\sup}$};
  }
  \node[font=\large] at (15.02,2.55) {$\cdots$};

  \silsmallunderbrace{-0.20}{1.00}{1}{1}
  \silsmallunderbrace{1.45}{2.65}{1}{2}
  \silsmallunderbrace{4.00}{5.20}{1}{k}
  \silsmallunderbrace{5.65}{6.85}{2}{1}
  \silsmallunderbrace{7.30}{8.50}{2}{2}
  \silsmallunderbrace{9.85}{11.05}{2}{k}
  \silsmallunderbrace{11.50}{12.70}{3}{1}
  \silsmallunderbrace{13.15}{14.35}{3}{2}
  \silsmallunderbrace{15.70}{16.90}{3}{k}

  \sillargeunderbrace{-0.22}{5.22}{1}
  \sillargeunderbrace{5.63}{11.07}{2}
  \sillargeunderbrace{11.48}{16.92}{3}
\end{tikzpicture}%
}}%
}
\caption{Example of an ordered graph $G$ and the corresponding $G'$ from the proof of the Sparse Incomparability Lemma \ref{thm:SparesIncomparabilityLemma}.}
\label{pic:SpraseIncLemmaGraphG}
\end{figure}

\textit{$G'$ is an ordered matching.}
Each vertex $w^{i_1j}_{i_2}$ can be adjacent only to $w^{i_2j}_{i_1}$ (when $\{v_{i_1}, v_{i_2}\} \in E(G)$), so every vertex of $G'$ has degree at most one.

\textit{Existence of $f\colon G' \to G$.}
Define $f(w) = v_i$ for all $w \in X_i$, $i \in [n]$. Since $X_1, \ldots, X_n$ are consecutive blocks in $G'$, the map $f$ preserves order. For any edge $\{w^{i_1j}_{i_2}w^{i_2j}_{i_1}\}$ of $G'$, we have $f(w^{i_1j}_{i_2}) = v_{i_1}$ and $f(w^{i_2j}_{i_1}) = v_{i_2}$, and $\{v_{i_1}, v_{i_2}\} \in E(G)$ by construction, so $f$ preserves edges.

\textit{Forward direction.}
If there exists an ordered homomorphism $h\colon G \to H$, then $h \circ f\colon G' \to H$ is an ordered homomorphism.

\textit{Reverse direction.}
Suppose $g\colon G' \to H$ is an ordered homomorphism with $|H| \le k$. We construct an ordered homomorphism $h\colon G \to H$.

Since $g$ is order-preserving, for every $u\in V(H)$ the fiber $g^{-1}(u)$ is an interval in the vertex order of $G'$. Hence $V(G')$ is partitioned into at most $|H|\le k$ consecutive intervals (the non-empty fibers), and therefore there are at most $k-1$ \emph{cuts} between consecutive vertices mapped to different vertices of $H$, where a cut is a pair of consecutive vertices of $G'$ that lie in different fibers.

We say that a sub-block $Y^j_i$ is \emph{split} if it contains at least one cut. Since the sub-blocks are pairwise disjoint, each cut can split at most one sub-block, so at most $k - 1$ sub-blocks are split in total.

We call $j \in [k]$ \emph{good} if $Y^j_i$ is unsplit for every $i \in [n]$. If no $j$ were good, then for each $j \in [k]$ there would exist some $i$ with $Y^j_i$ split. Since sub-blocks $Y^{j_1}_{i}$ and $Y^{j_2}_{i'}$ are disjoint whenever $j_1 \ne j_2$ (even for the same~$i$), these $k$ splits would lie in $k$ pairwise disjoint sub-blocks, giving at least $k$ split sub-blocks. This contradicts the bound of at most $k - 1$. Therefore, a good $j$ exists; fix one.

Since $j$ is good, no sub-block $Y_i^j$ contains a cut. As each fiber $g^{-1}(u)$ is an interval, this implies that every $Y_i^j$ lies within a single fiber, and hence $g$ is constant on $Y_i^j$; denote this value by $u_i\in V(H)$.
 Define $h\colon G \to H$ by $h(v_i) = u_i$ for all $i \in [n]$.

We verify that $h$ is an ordered homomorphism:
\begin{itemize}
    \item \emph{Order:} If $i_1 < i_2$, then every vertex of $Y^j_{i_1}$ precedes every vertex of $Y^j_{i_2}$ in $G'$ (since $X_{i_1}$ precedes $X_{i_2}$), so $h(v_{i_1}) = g(w) \le g(w') = h(v_{i_2})$ for any $w \in Y^j_{i_1}$, $w' \in Y^j_{i_2}$.

    \item \emph{Edges:} If $\{v_{i_1}, v_{i_2}\} \in E(G)$, then $\{w^{i_1j}_{i_2}w^{i_2j}_{i_1}\}$ is an edge of $G'$ with $w^{i_1j}_{i_2} \in Y^j_{i_1}$ and $w^{i_2j}_{i_1} \in Y^j_{i_2}$. Since $g$ is a homomorphism, $\{h(v_{i_1}), h(v_{i_2})\} = \{g(w^{i_1j}_{i_2}), g(w^{i_2j}_{i_1})\}$ is an edge of $H$.
\end{itemize}
Therefore $h\colon G \to H$ is an ordered homomorphism.
\end{proof}

\section{Order Density for Ordered Homomorphisms}
\label{sec:density}

As for the digraphs in~\cite{HellNesetrilGraphHomomorphisms}, let $\mathcal{G}$ denote the set of all ordered graphs, $G, H\in \mathcal{G}$, and let us write $G\le H$ for $G\to H$, and write $G<H$ for $G\to H$ and $H\not\to G$. We also see that ordered homomorphisms are transitive and reflexive relations on $\mathcal{G}$. However, $\le$ is in general not antisymmetric; therefore, $\le$ defines the quasi-order on $\mathcal{G}$.

Again, as in \cite{HellNesetrilGraphHomomorphisms}, we will transform quasiorder into partial order on $\mathcal{G}$, by choosing the ordered cores to be representatives for each equivalence class. We will denote by $\mathcal{C}$ the set of all non-isomorphic ordered cores, the set $\mathcal{C}$ is thus a partial order under $\le$.

We will then say that a partial order is \emph{dense} if for any $a<b$, there exists an element $c$ such that $a<c<b$. For graphs, we have either $K_2\le X$ or $X\le K_1$ for any ordered graph $X$, depending on whether $X$ has edges or not, resp. Thus, we see that the partial order $\mathcal{C}$ under $\le$ is not dense. We will therefore say that the ordered pair $(G,H)$ of ordered graphs $G, H\in \mathcal{C}, G<H$ forms a \emph{gap} in $\mathcal{C}$, if there is no ordered graph $F\in \mathcal{C}$, such that $G<F<H$. We may define a gap the same way for quasiorder $\mathcal{G}$ under $\le$. We aim to characterize the gaps.

Compared to graphs, we will show that the partial order $\mathcal{C}$ under $\le$ has many more gaps, but it is otherwise dense. But let us first start with an easy observation. We will notice that some of the following statements will hold even considering the ordered graphs of $\mathcal{G}$. Of course, if the pair of ordered graphs is a gap in quasiorder $\mathcal{G}$ under $\le$, then it is a gap in partial order $\mathcal{C}$ under $\le$.

Let a \emph{connected ordered graph} $G$ and an \emph{ordered component} of $G$ be defined in the same way as connected graphs and their components for unordered graphs, respectively. Then the following statement holds.

\begin{thm}
\label{thm:order}
    Let 
    $G_1$ be an arbitrary ordered graph 
    and $G_2$ be an ordered core, where every component of $G_2$ has more than two vertices, and $G_1<G_2$. Then there exists an ordered graph $F$ such that $G_1<F<G_2$.
\end{thm}

\begin{proof}

We will prove this statement by constructing such an ordered graph $F$ and showing that it satisfies all the required properties.

As we have shown in Lemma~\ref{thm:SparesIncomparabilityLemma}, when constructing an ordered matching $G'$ out of $G_2$, for any ordered graph $H$ on at most $k$ vertices there is an ordered homomorphism from $G'$ to $H$ if and only if there is an ordered homomorphism from $G_2$ to $H$. Let therefore $G_1$ be isomorphic to $H$. We see that because there is no ordered homomorphism from $G_2$ to $G_1$, there is no ordered homomorphism from $G'$ to $G_1$ due to the above equivalence.

Similarly, as in Lemma~\ref{thm:SparesIncomparabilityLemma}, we see that there exists an ordered homomorphism from $G'$ to $G_2$. Because $G_2$ is an ordered core and has all components on at least three vertices, there is no ordered homomorphism from $G_2$ to $G'$, since $G'$ is a matching.

Before constructing an ordered graph $F$ from our statement, let $v_1,v_2,\ldots,v_n$ be vertices in $G_2$ in their order. We then define the sets of vertices $W_i, i\in [n]$, where based on $f:G_1\to G_2$, each $W_i$ contains the set of vertices of $G_1$ that map to $v_i$ in $G_2$. Notice that some of $W_i$ can be empty.

Let us also take the sets $X_i, i\in [n]$ of $G'$ as defined in the proof of Lemma~\ref{thm:SparesIncomparabilityLemma}.

We then construct an ordered graph $F$ as a disjoint union of $W_i$ and $X_i, i\in [n]$ as follows:
\[
F = (W_1, X_1, W_2, X_2, \ldots, W_n, X_n).
\]

We will then preserve all edges in $F$ between vertices of $G_1$ in $W_i, i\in [n]$ and between vertices of $G'$ in $X_i, i\in [n]$, so that $\bigcup_i W_i = G_1$ and $\bigcup_i X_i = G'$. This completes the definition of $F$.

Notice that there is no edge between the vertices of $\bigcup_i W_i$ and $\bigcup_i X_i$, therefore, $F$ is not connected.

But then to show that there is no ordered homomorphism $h^{'}:G_2\to F$, we first observe that because $G_2$ is a core and it has only components of order larger than two, and because there is no edge between $\bigcup_i W_i$ and $\bigcup_i X_i$, and since $\bigcup_i X_i$ is an ordered matching, then if there is an ordered homomorphism $G_2\to F$, then $G_2$ could only map entirely to $\bigcup_i W_i$. But because $\bigcup_i W_i=G_1$ and there is no ordered homomorphism from $G_2$ to $G_1$, $G_2$ cannot map to $F$. 

Notice that if $G_2$ contained an isolated edge, then this edge could possibly map to $\bigcup_i X_i$ and the rest of $G_2$ could map to $\bigcup_i W_i$. This is why we need the assumption that the order of all components in $G_2$ is larger than two.

In order to show that there exists an ordered homomorphism $h:F\to G_2$, we first observe that for any fixed $i\in [n]$, we can glue $W_i$ and $X_i$. We can glue vertices within $W_i$, as if these were not independent, they could not map to the same vertex in $G_2$. We can also glue the vertices within $X_i$ as we showed in the proof of Lemma~\ref{thm:SparesIncomparabilityLemma}. Ultimately, we can glue the vertices in $W_i$ and $X_i$ altogether because there is no edge between $W_i$ and $X_i$. Therefore, let us construct an ordered graph $F^{'}$ by gluing $W_i$ and $X_i$ for every $i\in [n]$.

We see that there is an ordered homomorphism from $F$ to $F^{'}$, because as we reasoned before, for the ordered graph $F'$, resulting from gluing the vertices in $F$, there exists an ordered homomorphism $F\to F^{'}$. Therefore, we will show that there is an ordered homomorphism $h_1: F^{'}\to G_2$ in order to prove the existence of $h:F\to G_2$ (by transitivity).

We see that the number of vertices of $F^{'}$ and $G_2$ is the same, so $h_1$ will map the vertices of $w_i$ from $F^{'}$ to $v_i$ of $G_2, i\in [n]$. The order of vertices is, therefore, preserved by $h_1$.

We can see that if there is an edge $\{w_i,w_j\}$ in $F^{'}$, this edge originates from the presence of an edge between $W_i$ and $W_j$ or an edge between $X_i$ and $X_j, i,j \in [n]$. But if there is an edge between $W_i$ and $W_j, i,j \in [n]$, then this edge will be present in $G_2$, because of the existence of the ordered homomorphism $f$ from $\bigcup_i W_i=G_1$ to $G_2$, since $h_1$ maps the vertices of $G_1$ to the same vertices of $G_2$ as $f$. The same holds if there is an edge between $X_i$ and $X_j, i,j \in [n]$, since then this edge will be present in $G_2$, because of the existence of the ordered homomorphism $h_2$ from $\bigcup_i X_i=G'$ to $G_2$, since $h_1$ maps the vertices of $G'$ to the same vertices of $G_2$ as $h_2$.

We also see that there is an ordered homomorphism $g:G_1\to F$, since $F$ contains $\bigcup_i W_i=G_1$ and there is no ordered homomorphism $g':F\to G_1$, because there is no ordered homomorphism from $\bigcup_i X_i=G'$ to $G_1$. This completes the proof.
\end{proof}

Notice that Theorem~\ref{thm:order} does not address ordered matchings or ordered graphs containing isolated edges.

Therefore, the following statement then expands the class of ordered graphs for which the order defined by ordered homomorphisms is dense.

\begin{prop}
\label{prop:besticando}
    Let $G_1, G_2\in \mathcal{C}$, where $G_1$ is not isomorphic to $K_1$, $G_2$ contains at least one component that cannot be mapped to any of the components in $G_1$, and $G_1\le G_2$. Then there exists an ordered graph $F\in \mathcal{C}$ such that $G_1<F<G_2$.
\end{prop}

\begin{proof}
    We will again prove the statement by constructing an ordered graph $F$ that satisfies all the required properties.

    We will follow exactly the same construction of $F$, out of ordered cores $G_1, G_2\in \mathcal{C}$, as we did in the proof of Theorem~\ref{thm:order}, and we will adopt the same notation. Observe that since $G_1$ and $G_2$ are ordered cores, they do not contain isolated vertices.

    Since $G_1$ contains at least one edge, $G_2$ needs to contain at least one component of order larger than two, therefore $G_2$ cannot be an ordered matching, and, of course, there is no ordered homomorphism $G_2\to G_1$.

    We also immediately see that $G_1\le F$, since $F$ again contains $G_1$.

    It is also clear that using the reasoning in the proof of Lemma~\ref{thm:SparesIncomparabilityLemma}, there is no ordered homomorphism from $F$ to $G_1$, since $F$ contains $G'$, from which there is no homomorphism to $G_1$. Therefore, $G_1< F$.

    We can also see that using the same arguments as in the proof of Theorem~\ref{thm:order}, there is an ordered homomorphism from $F$ to $G_2$, since by gluing each $W_i$ and $X_i, i=1,\ldots,|G_2|$ we again get an ordered graph $F'$, such that $F\to F'\to G_2$.

    The only difference compared to Theorem~\ref{thm:order} is to prove that there is no ordered homomorphism from $G_2$ to $F$. But because there exists a component $G_2'\subseteq G_2$, such that $G_2'$ cannot be mapped to any of the components in $\bigcup_i W_i=G_1$ in $F$ and this component is of order greater than two (by the reasoning above), $G_2'\subseteq G_2$ cannot map to $G_1\subset F$ nor to $G'\subset F$, respectively. Therefore, there is no ordered homomorphism $G_2\to F$. This completes the proof.
\end{proof}

We see that the ordered graphs forming a dense order in Proposition~\ref{prop:besticando} can contain ordered cores with components of order two. However, it is not a superset of ordered graphs defined by Theorem~\ref{thm:order}, since if $G_1$ in Proposition~\ref{prop:besticando} contains a component of large order, e.g.\ $K_m$ with large $m$, then $G_2$ needs to contain a component that cannot map to $K_m$. This is not a condition for ordered graphs $G_2$ in Theorem~\ref{thm:order}, where it is sufficient that all components of $G_2$ are of order larger than two.

Based on the previous findings (or rather the classes of ordered graphs for which we were not able to prove the dense order), we will continue by having a closer look at the ordered monotone matchings. Using Theorem~\ref{thm:Uniq}, we get the following corollary.

\begin{cor}
\label{cor:M1M2order}
    Let $M_1$ and $M_2$ be ordered monotone matchings. Then $(M_1,M_2)$ is a gap in quasiorder $\mathcal{G}$ under $\le$.
\end{cor}

\begin{proof}
    Assume for contradiction, that there exists $F\in \mathcal{G}$ such that $M_1<F<M_2$. Theorem~\ref{thm:Uniq} tells us that if there is no $h^{'}:M_2\to F$, then there is $g'':F\to K_2$. But, since $K_2=M_1$, there exists $g':F\to M_1$, a contradiction.
\end{proof}

In fact, the whole class of (consecutive) monotone matchings forms a gap in the quasiorder $\mathcal{G}$ under $\le$.

\begin{prop}
\label{prop:MonMatchOrder}
    Let $M_i$ and $M_{i+1}, i\in \mathbb{N}$ be ordered monotone matchings. Then $(M_i, M_{i+1})$ forms a gap in quasiorder $\mathcal{G}$ under $\le$.
\end{prop}

\begin{proof}
    Let us again assume for contradiction, that there exists an ordered core $F\in \mathcal{G}$, such that $M_i<F<M_{i+1}$. Using again Theorem~\ref{thm:Uniq}, we first observe that because there is no $h^{'}:M_{i+1}\to F$, then there is $h^{''}:F\to K_{i+1}$. Using the same Theorem~\ref{thm:Uniq} again, we also observe that because there is $g:M_i\to F$, then there is no $g'':F\to K_i$. But because $h:F\to M_{i+1}$, and $h^{''}:F\to K_{i+1}$, and because there is no $g'':F\to K_i$, then the ordered graph $F$ must consist of precisely $i$ (connected or disconnected) subgraphs that map to $i$ doubles of $M_{i+1}$. 
    
    More precisely, $F$ maps surjectively to the maximum $i$ components of $M_{i+1}$ because $h:F\to M_{i+1}$ and $h^{''}:F\to K_{i+1}$. In other words, if $F$ consisted of $i+1$ subgraphs mapped surjectively to the $i+1$ components of $M_{i+1}$, then it cannot map to $K_{i+1}$. Also, $F$ maps to no less than $i$ components of $M_{i+1}$ because there is no $g'':F\to K_i$. I.e., if $F$ mapped to less than $i$ components of $M_{i+1}$ then there would exist $g'':F\to K_i$. 
    
    But then $i$ components of $M_{i+1}$ correspond to an ordered graph isomorphic to $M_{i}$. Therefore there must be an ordered homomorphism $g':F\to M_i$, a contradiction.
\end{proof}

We now show that the gaps are formed by even more generic pairs of ordered graphs. Again, we will denote disjoint union of two ordered graphs by $+$.

\begin{thm}
\label{thm:gapGe}
Let $G_1,G_2\in\mathcal C$, $G_2=G_1+e$, where $e$ is an isolated edge, and $G_1<G_2$. Let $B$ be the ordered subgraph of $G_1$ induced by the vertices belonging to connected components of order at least three. If $B$ is empty or is an ordered core, then $(G_1,G_2)$ is a gap in partial order $\mathcal C$ under ``$\le$''.
\end{thm}

\begin{proof}
Every endomorphism of an ordered core is the identity, and a core other than $K_1$ has no isolated vertices. Here $G_1\ne K_1$, since otherwise $G_1+e$ would not be a core.

We use the following observation. If a homomorphism from an ordered core $X$ to $Y$ maps an edge $ab$ to an isolated edge $uv$, collapse the fibres over $u,v$ to $a,b$, respectively, and fix all other vertices. The fibres are intervals and every edge incident with either fibre joins the two fibres, so this defines an endomorphism of $X$. Consequently both fibres are singletons and $ab$ is an isolated edge.

For any $f:G_1\to G_2$, the observation applied to $B$, when nonempty, shows that $f(B)\subseteq B$: no component of $B$ can meet an isolated edge of $G_2$. Hence $f$ fixes $B$. The image $J=f(G_1)$ consists of $B$ and isolated edges of $G_2$. Fixing $B$ and choosing a preimage edge for each of the latter defines an ordered homomorphism $t:J\to G_1$ with $f\circ t=\mathrm{id}_J$; representatives of distinct fibres retain their order. Since $t\circ f=\mathrm{id}_{G_1}$ by corehood, $f$ is an isomorphism onto $J$. As $G_2$ has two more vertices than $G_1$, we have $J=G_2-V(d)$ for a single isolated edge $d$.

Now consider $G_1\xrightarrow{\alpha}F\xrightarrow{\beta}G_2$ and apply the preceding argument to $f=\beta\circ\alpha$. If $d$ has no preimage edge under $\beta$, then $\beta$ maps $F$ into $G_2-d$, retaining the endpoints of $d$. This graph retracts onto $J$: fix $J$ and map each isolated endpoint to its largest predecessor in $J$, or its first vertex if there is no predecessor. Thus $F\to J\cong G_1$. Otherwise choose an edge of $F$ mapping to $d$ and combine it with $\alpha\circ f^{-1}:J\to F$ to obtain $G_2\to F$. This preserves order because all selected vertices lie in their respective $\beta$-fibres. Thus no strict intermediate exists.
\end{proof}

It remains to study the density of the order induced by ordered homomorphisms on ordered graphs that contain isolated edges. A natural starting point is the class of ordered matchings, as the simplest examples of such graphs.

We see that the role of ordered matchings is central in the area of ordered graphs and their ordered homomorphisms. In \cite{certik_complexity_2025,certik_matching_2025, certikorderedgraphsphd2025}, we show that this also applies to the computational complexity of various problems related to ordered homomorphisms.


\end{document}